\documentclass[12pt]{amsart}
\usepackage{amssymb,verbatim,enumerate,ifthen}
\usepackage{cite}
\usepackage[mathscr]{eucal}
\usepackage[utf8]{inputenc}
\usepackage[T1]{fontenc}
\usepackage{esint}
\usepackage{enumerate}
\usepackage[marginparwidth=2.4cm]{geometry}
\newtheorem{thm}{Theorem}[section]
\newtheorem{cor}[thm]{Corollary}
\newtheorem{prop}[thm]{Proposition}
\newtheorem{lem}[thm]{Lemma}

\newtheorem{Athm}{Theorem}

\theoremstyle{definition}

\newtheorem{rem}{Remark}

\newtheorem*{xrem}{Remark}
\newcommand{\norm}[1]{\left\|#1\right\|}

\numberwithin{equation}{section}
\def\eq#1{{\rm(\ref{#1})}}
\def\Eq#1#2{\ifthenelse{\equal{#1}{*}}
  {\begin{equation*}\begin{aligned}[]#2\end{aligned}\end{equation*}}
  {\begin{equation}\begin{aligned}[]\label{#1}#2\end{aligned}\end{equation}}}
\newcommand{\Res}{\mathcal{R}}
\newcommand\QA[1]{\mathscr{A}_{#1}}
\newcommand{\restr}{\!\!\upharpoonright}

\def\B{\mathscr{B}}
\def\D{\mathscr{D}}

\def\G{\mathscr{G}}
\def\H{\mathscr{H}}

\newcommand\EE{{\bf E}}

\def\MM{\mathbf{M}}
\def\KK{\mathbf{K}}
      
\def\vone{{\bf1}}

\newcommand\R{\mathbb{R}}
\newcommand\N{\mathbb{N}}

\DeclareMathOperator{\Var}{Var}
\DeclareMathOperator{\sign}{sign}
\newcommand\abs[1]{\left|#1\right|}

\author{Pawe{\l} Pasteczka}
\address{Institute of Mathematics\\ University of Rzesz\'ow\\ Pigonia 1\\ 35-310 Rzesz\'ow\\ Poland}
\email{ppasteczka@ur.edu.pl}

\title[On the new smoothness class of means...]{On the new smoothness class of means and its impact to mean-type mappings}
\begin{document}
\begin{abstract}
We define so-called residual means, which have a Taylor expansion of the form
$M(x)=\bar x +\tfrac12 \xi_M(\bar x) \Var(x)+o(\|x-\bar x\|^\alpha)$ for some $\alpha>2$ and a single-variable function $\xi_M$ ($\bar x$~stands for the arithmetic mean of the vector $x$), and show that all symmetric means which are three times continuously differentiable are residual. We also calculate the value of residuum for quasideviation means and a few subclasses of this family.

Later, we apply it to establish the limit of the sequence $\big(\frac{\text{Var}\ {\bf M}^{n+1}(x)}{(\text{Var}\ {\bf M}^n(x))^2}\big)_{n=1}^\infty$, where ${\bf M} \colon I^p\to I^p$ is a mean-type mapping consisting of $p$-variable residual means on an interval $I$, and $x \in I^p$ is a nonconstant vector.
\end{abstract}

\maketitle
\section{Introduction}

Various aspects of means have been continuously studied for more than a hundred years. There is a wide variety of problems related to them, such as comparability, equality between means, invariance, etc. The value of means close to the diagonal (that is, the set where all entries are equal to each other) plays an important role in this theory. Let us just recall a classical result stating that if two quasiarithmetic means (see section~\ref{sec:QA} for the definition) are comparable in the vicinity of the diagonal, then they are comparable on the whole domain (see, for example, Hardy-Littlewood-Poly\'a \cite{HarLitPol34}). Moreover, under the additional smoothness assumptions, comparability of two quasiarithmetic means can be expressed in terms of comparability of two single-variable functions. More precisely, if $f,g\colon I \to \R$ satisfy the natural smoothness assumptions, then $\QA{f}\le \QA{g}$ if and only if $\tfrac{f''}{f'}\le \tfrac{g''}{g'}$ ($\QA{f}$ stands for the quasiarithmetic mean generated by $f$). This fact is easily implied by Jensen's inequality; see Mikusiński \cite{Mik48} for the most general wording of this statement. 

Furthermore, it is known that iterations of mean-type mapping consisting of continuous strict means converge to some point on the diagonal (see, for example, \cite[Theorem~8.8]{BorBor87}). In 2018 it was proved \cite{Pas18b} that in the case when all means are quasiarithmetic, the precise speed of convergence is related to the variance of Arrow-Pratt indexes of all these means. The precise result reads as follows.
\begin{Athm}\label{thm:A}
 Let $p \in \N$, $I$ be an interval, $f_1,\dots,f_p \colon I \to \R$ be strictly monotone, twice continuously differentiable functions such that each $f_i''$ is locally Lipschitz. Define the mean-type mapping $\MM \colon I^p \to I^p$ by 
$\MM(x):=(\QA{f_1}(x),\dots,\QA{f_p}(x))$,
 where $\QA{f_i}$ is the quasiarithmetic mean generated by $f_i$. 
 
 For all $x \in I^p$ either $\MM^n(x)$ is constant for some $n \in \N$ or
 \Eq{*}{
 \lim_{n \to \infty} \frac{\Var \MM^{n+1}(x)}{\big(\Var \MM^n(x)\big)^2}=\frac14 \Var\bigg(\frac{f_1''(m)}{f_1'(m)},\dots,\frac{f_p''(m)}{f_p'(m)}\bigg),
 }
 where $m$ is the value of the unique $\MM$-invariant mean at $x$.
\end{Athm}

This theorem supports the idea of searching the generalization of the Arrow-Pratt index for other families of means. The comparability-type result will no longer be valid since the comparability of means is not localizable in general. On the other hand, there is the hope to reprove Theorem~\ref{thm:A} for the broader class of means. Our motivation comes from the theorem of Borwein-Borwein.
\begin{Athm}[\!\!\cite{BorBor87}, Theorem 8.8]\label{thm:BorBor}
    Let $p \in \N$, $I$ be an interval, $M_1\dots,M_p\colon I^p\to I$ be strict means, $x \in I^p$, and $\MM:=(M_1,\dots,M_p)\colon I^p\to I^p$. If all $M_i$-s are continuously differentiable and $\MM^{n}(x)$ is nonconstant for all $n \in \N$ then
    \Eq{*}{
    \lim_{n\to \infty} \frac{\max \MM^{n+1}(x)-\min \MM^{n+1}(x)}{\max \MM^{n}(x)-\min \MM^{n}(x)}=0.
    }
    If the means are twice continuously differentiable, convergence in the Gaussian iteration is quadratic (uniformly on compact subsets).
\end{Athm}

Based on Theorem~\ref{thm:BorBor}, the most natural assumption in a potential generalization of Theorem~\ref{thm:A} would be the twice continuous differentiability of the means. Regretfully, our assumptions are more restrictive. We will deal with this in the next section. We define so-called >>residual<< means and show that all means which are three times continuously differentiable are residual (section~\ref{sec:Avm}).
Then the Arrow-Pratt index corresponds to the  >>residuum<< of a mean. Later, after some auxiliary results, we generalize Theorem~\ref{thm:A} for mean-type mappings consisting of residual means (section~\ref{sec:ApplicationInvMeans}). In the final section, we establish the value of residuum for quasideviation means and a few subclasses of this family.

\section{\label{sec:Avm}Approximate values of means}
This section, in a rough sketch, is devoted to estimating the value of $p$-variable $\mathcal{C}^2$ means $M \colon I^p \to I$ close to the diagonal (Lemma~\ref{lem:1}). Prior to this, let us formally introduce a few preliminary definitions.

Throughout this note $p \in \N$ and $I$ is an open subinterval of $\R$. Recall that a \emph{\mbox{$p$-variable} mean on $I$} is an arbitrary function $M \colon I^p\to I$ satisfying the \emph{mean property}, that is
\Eq{*}{
\min(x_1,\dots,x_p)\le M(x_1,\dots,x_p)\le\max(x_1,\dots,x_p)
\text{ for all }x_1,\dots,x_p \in I.
}
A mean $M$ is called \emph{symmetric} whenever $M(x\circ \sigma)=M(x)$ for all $x\in I^p$ and a permutation $\sigma\colon\{1,\dots,p\}\to\{1,\dots,p\}$. Since $M$ is a real-valued function on $I^p$ we can easily define its properties such as continuity and other smoothness properties. Furthermore, we say that a mean $M$ is \emph{strict} provided
\Eq{*}{
\min(x_1,\dots,x_p)< M(x_1,\dots,x_p)<\max(x_1,\dots,x_p)
}
for every non-constant vector $(x_1,\dots,x_p) \in I^p$.

From now on, let $\vone$ stands for the vector of the suitable dimension with all elements equal to one. Observe that the mean property forces the values of the mean on the diagonal
\Eq{*}{
\Delta_p(I):=\{(x_1,\dots,x_p)\in I^p\colon x_1=\cdots=x_p\}=\{x\vone\colon x \in I\}.
}
More precisely, $M(x{\bf 1})=x$ for all $x \in I$ (this property alone is known as reflexivity). 

In this paper, we will also use local properties of means. This concept goes back to P\'ales--Zakaria \cite{PalZak17}. That is, we say that a mean $M \colon I^p \to I$ is \emph{locally $\mathcal{C}^k$} if it is $\mathcal{C}^k$ in an open neighborhood of $\Delta_p(I)$.

We are heading towards the lemma, which establishes the value of a mean close to the diagonal.
Later we introduce the family of means satisfying the stronger version of the inequality proved in this lemma (residual means). Then, in Theorem~\ref{thm:main}, we show that all locally $\mathcal{C}^3$ means admit this property. The motivation for such an approach will be clarified in the next section.

For a vector $s \in \R^p$ let us introduce three classical, brief notations
\Eq{*}{
\EE s:=\frac1p \sum_{i=1}^p s_i;\qquad
\EE s^2:= \frac1p \sum_{i=1}^p s_i^2; \qquad \Var(s):=\EE s^2-(\EE s)^2.
}
Now we are ready to proceed to the first lemma.
\begin{lem}\label{lem:1}
Let $M\colon I^p \to I$ be a symmetric, locally $\mathcal{C}^2$ mean.
There exists a continuous function $\xi_M \colon I \to I$ such that for all $x \in I$ and $s\in \R^p$ we have
\Eq{E1}{
M(x\vone+t s)&=  x+t\EE s+\frac{t^2}2\xi_M(x) \Var(s)+o(t^2) \quad \text{ for }t\approx 0.
}
Moreover, for all $x \in I$,
\Eq{xiM}{
\xi_M(x)=-p^2 \partial_1\partial_2M(x\vone);\qquad\xi_M(x)=\frac{p^2}{p-1} \partial_1^2M(x\vone).
}
\end{lem}

\begin{proof} Fix $x \in I$ and $s \in \R^p$ arbitrarily. First, we use Taylor's theorem. For $t$ close to zero, we get
\Eq{118}{
M(x\vone+t s)=M(x\vone)+t \sum_{i=1}^p \partial_iM(x\vone)s_i+\frac{t^2}2 \sum_{i=1}^p\sum_{j=1}^p \partial_i\partial_j M(x\vone) s_is_j+o(t^2).
}

Since $M$ is symmetric, we know that the value of the partial derivatives $\partial_iM$ taken on the diagonal does not depend on $i=1,\dots,p$. The same holds for second partial derivatives (separately pure and mixed). Therefore, for the sake of brevity, let us define $\alpha,\beta,\gamma \colon I \to \R$ by
\Eq{*}{
\alpha(x):=\partial_1^2M(x\vone);\quad
\beta(x):=\partial_1\partial_2M(x\vone);\quad \gamma(x):=\partial_1M(x\vone).
}

Since $M$ is reflexive, we have $M(x\vone)=x$ and $M((x+t)\vone)=x+t$. Now we take the first order Taylor's expansion of $M((x+t)\vone)$ with respect to $t$ in an open neighborhood of zero. Then we get 
\Eq{*}{
x+t=M\big((x+t)\vone\big)=M(x\vone)+t \sum_{i=1}^p \partial_iM(x\vone)+o(t)&=x+tp\gamma(x)+o(t)\\
&\text{ for }x \in I\text{ and }t \approx 0.
}
Now we apply a few transformations side-by-side. First, we subtract $x$, then divide by $tp$ and, finally, take a limit $t \to 0$. Then we get  $\gamma \equiv \frac1p$. Therefore \eq{118} for $t \approx 0$ has the following form
\Eq{E179}{
M(x\vone+t s)&=x+t \sum_{i=1}^p \frac{s_i}p+\frac{t^2}2 \Big( \alpha(x) \sum_{i=1}^p s_i^2+\beta(x)\sum_{i\ne j}s_is_j \Big)+o(t^2).
} 
In a particular case $s=\bf 1$, this identity yields
\Eq{*}{
x+t =M(x\vone+t \vone)&=x+t+ \tfrac{t^2}2\big(p\alpha(x) + ( p^2 -p)\beta(x)
\big)+o(t^2).
}
Now, repeating the reasoning above, we subtract $(x+t)$ side-by-side, divide by $\frac12 t^2$, and take the limit as $t\to0$. Then we have $p\alpha(x)+ ( p^2 -p)\beta(x)=0$. After easy simplification, we obtain $\alpha(x)=(1-p)\beta(x)$.

Furthermore, using the standard identity for symmetric polynomials of degree two, we get
\Eq{*}{
\sum_{i\ne j} s_is_j&= \Big(\sum_{i=1}^p s_i\Big)^2-\sum_{i=1}^p s_i^2= p^2 (\EE s)^2-p\EE s^2.
}
Thus, by \eq{E179}, we obtain
\Eq{E156}{
M(x\vone+t s)&=x+t \EE s+ \frac{t^2}{2}\Big(\alpha(x) p\EE s^2+ \beta(x)\big( p^2 (\EE s)^2-p\EE s^2\big)
\Big)+o(t^2).
}
Now, let us simplify
\Eq{*}{
\alpha(x) p\EE s^2+ \beta(x)\big( p^2 (\EE s)^2-p\EE s^2\big)&=
p(1-p)\beta(x)\EE s^2+ \beta(x)\big( p^2 (\EE s)^2-p\EE s^2\big)\\
&=\big((p(1-p)-p) \EE s^2+ p^2(\EE s)^2\big)\beta(x)
\\
&=-p^2\beta(x)\big( \EE s^2-(\EE s)^2
\big).
}
Then the equality above jointly with \eq{E156} implies \eq{E1} with the substitution
$\xi_M(x):=-p^2\beta(x)$. Applying the definition of $\beta$, we obtain the first equality in \eq{xiM}. Similarly
\Eq{*}{
\xi_M(x)=-p^2\beta(x)=\frac{-p^2}{1-p}\alpha(x)=\frac{p^2}{p-1}\partial_1^2M(x{\bf 1}),
}
which is the second form of the function $\xi_M$ appearing in \eq{xiM}.
\end{proof}

Based on the above lemma, for every symmetric mean $M\colon I^p \to I$ which is locally $\mathcal{C}^2$,
the function $\xi_M\colon I \to I$ defined by \eq{xiM} is called the \emph{residuum of $M$}. In the following technical lemma, we show another way to calculate the value of residuum.
\begin{lem}\label{lem:Formxi_M}
    Let $M\colon I^p \to I$ be a symmetric, locally $\mathcal{C}^2$ mean. Then 
    \Eq{*}{
    \xi_M(x)=\lim_{t\to 0} \frac{2p^2}{(p-1)t^2}\big(M(x+t,x,\dots,x))-x-\tfrac tp\big)\qquad (x\in I).
    }
\end{lem}
\begin{proof}
    Define $s:=(1,0,0,\dots)$. Then $\Var(s)=\frac{p-1}{p^2}$. Thus, by Lemma~\ref{lem:1} we get
    \Eq{*}{
    \xi_M(x)=\lim_{t \to 0}\frac{2\big(M(x\vone+t s)-x-t\EE s\big)}{t^2\Var(s)}
    =\lim_{t \to 0}\frac{2p^2}{p-1}\frac{M(x\vone+t s)-x-\frac t p}{t^2},
    }
    which completes the proof.
\end{proof}

\subsection{\label{sec:residual} Residual means}
Prior to defining residual means, let us fix a norm $\norm{\cdot}$ in $\R^p$ to be the $\ell^\infty$ norm. Hereafter, we will use two standard notations
\Eq{*}{S^{p-1}:=\{x \in \R^p \colon \norm{x}=1\},\quad\text{ and }\quad\mathcal{B}^p(\varepsilon):=\{s \in \R^p \colon \norm{s}\le\varepsilon\}\text{ for }\varepsilon>0.
}
We stress that none of our results will depend on the choice of the norm. On the other hand, the norm could affect the constants indicated in their proofs.

We say that a symmetric mean $M \colon I^p \to I$ is \emph{residual} if it is locally $\mathcal{C}^2$ and for every compact subinterval $J \subset I$  there exist  $\lambda,\varepsilon \in(0,+\infty)$ and $\alpha \in (2,+\infty)$ such that
\Eq{D-res}{
\abs{M(x\vone+s)-x-\EE s-\tfrac12 \xi_M(x) \Var(s)}\le\lambda\norm{s}^\alpha
}
is valid for all $x \in J$ and $s\in \mathcal{B}^p(\varepsilon)$,
where $\xi_M$ is given by \eq{xiM}.

\begin{rem}
Note that the values $\lambda,\varepsilon$ and $\alpha$ depend only on $M$, $J$ and the norm. Thus, purely formally, we assume that for every compact subinterval $J$ of $I$ there exists $\lambda=\lambda_{M,J} \in (0,+\infty)$, $\varepsilon=\varepsilon_{M,J} \in (0,+\infty)$ and $\alpha=\alpha_{M,J} \in (2,+\infty)$ such that \eq{D-res} holds for all $x \in J$ and $s \in \mathcal{B}^p(\varepsilon)=\mathcal{B}^p(\varepsilon_{M,J})$. 
\end{rem}

\begin{rem}
The inequality in \eq{D-res} cannot be replaced by a strict inequality. Indeed, for a null vector $s$, it simplifies to $0 \le 0$.
\end{rem}

\medskip 

Let $\Res_p(I)$ be the family of all residual $p$-variable means on $I$. Clearly, every residual mean is locally $\mathcal{C}^2$. We show a sort of reverse statement. 

\begin{thm}\label{thm:main}
Every symmetric, locally $\mathcal{C}^3$  mean $M \colon I^p\to I$ is residual.
\end{thm}
\begin{proof}
Let $M \colon I^p\to I$ be a symmetric, locally $\mathcal{C}^3$  mean. Take a compact subinterval $J$ of $I$ and choose $\varepsilon>0$ such that $[\inf J -2\varepsilon, \sup J+2\varepsilon]\subset I$ and $M$ is $\mathcal{C}^3$ on $\{x \in J^p \colon \max(x)-\min(x)<3\varepsilon \}$.  For the sake of brevity, let us define a set $\Gamma:=J \times S^{p-1} \times [-\varepsilon,\varepsilon]$. In view of Taylor's theorem, there exist continuous functions $A, B \colon J \times S^{p-1} \to \R$ and a function $C \colon \Gamma \to \R$ such that
    \Eq{*}{
        M(x\vone+t \eta)=M(x{\bf 1})+A(x,\eta) t +B(x,\eta) t^2+C(x,\eta,t)t^3\quad\text{ for all }(x,\eta,t)\in \Gamma.
        }
 Functions $A$ and $B$ are defined by
\Eq{*}{
A(x,\eta)&:=\phantom{1} \partial^{\phantom2}_t M(x\vone+t \eta) \big|_{t=0};\\
B(x,\eta)&:=\tfrac12 \partial^2_t M(x\vone+t \eta) \big|_{t=0}.
}
The value of $C$ is determined by $A$, $B$ and equals
\Eq{*}{
C(x,\eta,t):=
\begin{cases}
\dfrac{M(x\vone+t \eta)-(M(x{\bf 1})+A(x,\eta) t +B(x,\eta) t^2)}{t^3}& \text{ for }t \ne 0,\\[3mm]
\frac16 \partial^3_t M(x\vone+t \eta) \big|_{t=0}& \text{ for }t =0.
\end{cases}
}

Let us now define the Taylor's reminder term $R\colon \Gamma\to \R$ by
\Eq{*}{
R(x,\eta,u):=\tfrac16 \partial^3_t M(x\vone+t \eta)\big|_{t=u}.
}
Clearly $M$ is three times continuously differentiable on the set $\{x\vone+t\eta\colon (x,\eta,t)\in \Gamma\}$. Therefore, by Taylor's theorem, for every $(x,\eta,t) \in \Gamma$ there exists $u\in [-\varepsilon,\varepsilon]$ such that $C(x,\eta,t)=R(x,\eta,u)$. 
Since $R$ is a continuous function defined on a compact set, the value $\lambda:=\sup\{|R(x,\eta,u)|\colon(x,\eta,u)\in \Gamma\}$ is finite. Therefore 
\Eq{*}{
\sup \big\{|C(x,\eta,t)|\colon(x,\eta,t)\in \Gamma\big\} \le \sup \big\{|R(x,\eta,u)|\colon(x,\eta,u)\in \Gamma\big\} = \lambda
} 
and we get
\Eq{E:294}{
\big|M(x\vone+t \eta)-(M(x{\bf 1})+A(x,\eta) t &+B(x,\eta) t^2)\big| \le \lambda |t|^3 \text{ for all }(x,\eta,t) \in \Gamma.
}
However, in view of Lemma~\ref{lem:1}, we get $M(x{\bf 1})=x$, $A(x,\eta)=\EE \eta$, and $B(x,\eta)=\tfrac12 \xi_M(x) \Var(\eta)$. Therefore \eq{E:294} simplifies to
\Eq{*}{
\abs{M(x\vone+t \eta)-x-\EE (t\eta)-\tfrac{1}2 \xi_M(x) \Var(t\eta)}\le\lambda \norm{t\eta}^3 \text{ for all }(x,\eta,t)\in \Gamma.
}

Finally, let us observe that every vector $s\in \mathcal{B}^p(\varepsilon)$ can be expressed as a product $s=t\eta$ where $t \in [-\varepsilon,\varepsilon]$ and $\eta\in S^{p-1}$. Therefore
\Eq{*}{
\abs{M(x\vone+s)-x-\EE s-\tfrac{1}2 \xi_M(x) \Var(s)}\le\lambda \norm{s}^3
}
for all $x\in J$ and $s\in \mathcal{B}^p(\varepsilon)$, which is precisely the definition of residuality (with $\alpha=3$).
\end{proof}

\begin{rem}
In this way, we have proved that residuality is in fact a local, intermediate smoothness assumption between being locally $\mathcal{C}^2$ and locally $\mathcal{C}^3$. Indeed, we know that every residual mean is (by the definition) locally $\mathcal{C}^2$. Conversely, by Theorem~\ref{thm:main} every symmetric mean which is locally $\mathcal{C}^3$ is residual.

We conjecture that the residuality with $\alpha \in (2,3)$ refers to the H\"older continuity of the second derivative. 
\end{rem}
\subsection{Means on any number of arguments}
Until now, all means considered in this note were defined only for a given (fixed) number of arguments. It is, however, quite often that the input to the mean can be a vector of an arbitrary (finite) length. Then a mean on $I$ is a function $M\colon \bigcup_{p=1}^\infty I^p \to I$ such that each $M\restr_p:=M|_{I^p}$ is a $p$-variable mean on $I$. Then we use the standard convention that such a mean has a certain property (for example, continuity, symmetry, etc.) precisely if all its restrictions $M\restr_p$ admit this property. In the same spirit, we say that a mean $M\colon \bigcup_{p=1}^\infty I^p \to I$
is residual if, for all $p\in \N$, its $p$-variable restriction $M\restr_p$ is residual. 

We also need one property that binds values of a mean for various numbers of parameters. Namely, mean $M$ is called \emph{repetition invariant} if, for all $p,q\in\N$
and $(x_1,\dots,x_p)\in I^p$, the following identity is satisfied
\Eq{*}{
  M(\underbrace{x_1,\dots,x_1}_{q\text{-times}},\dots,\underbrace{x_p,\dots,x_p}_{q\text{-times}})
   =M(x_1,\dots,x_p).
}
This definition was first introduced by P\'ales-Pasteczka \cite{PalPas16}. It turns out that this property holds for a number of classical families of means -- for example, a broad class of quasideviation means (see section~\ref{sec:QD} below). 

In the next lemma, we show that this property also plays an important role in our new concept of residual means.

\begin{thm}\label{lem:2}
Let $M\colon \bigcup_{p=1}^\infty I^p \to I$ be a symmetric, repetition invariant, residual mean. Then the value of $\xi_{M\,\restr_p}$ does not depend on $p$.
\end{thm}
\begin{proof}
    Let $p,q \in \N$, $p,q \ge 2$. First, we prove the equality $\xi_{M\,\restr_{pq}}=\xi_{M\,\restr_p}$. To keep the proof more compact, fix $x \in I$ and define 
    \Eq{*}{
    \hat x:=(\underbrace{x,\dots,x}_{p\text{ times}}),\quad 
    \hat s:=(1,\underbrace{0,\dots,0}_{(p-1)\text{ times}}),\quad 
    \bar x:=(\underbrace{x,\dots,x}_{pq\text{ times}}),\quad 
    \bar s:=(\underbrace{1,\dots,1}_{q\text{ times}},\underbrace{0,\dots,0}_{(pq-q)\text{ times}}).\quad 
    }
    Then $\bar s$ and $\hat s$ possess the same probability distribution (provided choosing entries are uniformly distributed), and whence $\EE\hat s=\EE\bar s$, and $\Var(\hat s)=\Var(\bar s)$. 
     Furthermore, since $M$ is repetition invariant, we have $M(\bar x+t\bar s)=M(\hat x+t\hat s)$ for all admissible $t$.
    Thus, Lemma~\ref{lem:1} implies (here and below $t\approx0$)
    \Eq{*}{
    0&=M(\bar x+t\bar s)-M(\hat x+t\hat s)\\
    &= x+t\EE \bar s+\xi_{M\,\restr_{pq}}(x) \Var(\bar s)\tfrac{t^2}2-x-t\EE \hat s- \xi_{M\,\restr_{p}}(x) \Var(\hat s)\tfrac{t^2}2+o(t^2)\\
    &=\tfrac{t^2}2\Var(\hat s)(\xi_{M\,\restr_{pq}}(x)- \xi_{M\,\restr_{p}}(x))+o(t^2).
    }
Since $\hat s$ is a nonconstant vector we know that $\Var(\hat s)\ne 0$. Thus, we can divide by $\frac{t^2}2\Var(\hat s)$ and take the limit $t \to 0$, to obtain $0=\xi_{M\,\restr_{pq}}(x)-\xi_{M\,\restr_{p}}(x)$. Since $x$ was chosen to be an arbitrary element in $I$, we immediately get $ \xi_{M\,\restr_{pq}}=\xi_{M\,\restr_{p}}$. 

Finally we get $\xi_{M\,\restr_{p}}=\xi_{M\,\restr_{pq}}=\xi_{M\,\restr_{qp}}=\xi_{M\,\restr_{q}}$, which completes the proof.
\end{proof}
\section{An application to invariant means\label{sec:ApplicationInvMeans}}

A mapping $\mathbf{M}\colon I^{p}\rightarrow I^{p}$ is referred to as 
\emph{mean-type} if $\mathbf{M}=(M_{1},\dots ,M_{p})$ for some means $M_1,\dots,M_p\colon I^p\rightarrow I$.
A mean $K\colon I^{p}\rightarrow I$ is \emph{invariant with respect to the mean-type mapping $\mathbf{M}$} (briefly \emph{$\mathbf{M}$-invariant}), if it solves the functional equation $K\circ 
\mathbf{M}=K$.

Invariance property is a vital aspect of the theory of means. There are two classical studies, Lagrange \cite{Lag84} and Gauss \cite{Gau18}, which could be considered as the beginning of this research. It has been extensively studied by many authors since then. For example J.~M.~Borwein and P.~B.~Borwein \cite{BorBor87} extended some earlier ideas \cite{FosPhi84a,Leh71,Sch82} and 
generalized the original iteration to a vector of continuous, strict means of an arbitrary length.
For several recent results about the Gaussian product of means, see the papers by Baj\'ak--P\'ales 
\cite{BajPal09b,BajPal09a,BajPal10,BajPal13}, by Dar\'oczy--P\'ales \cite{Dar05a,DarPal02c,DarPal03a}, 
by G{\l}azowska \cite{Gla11b,Gla11a}, by Jarczyk--Jarczyk \cite{JarJar18}, by Matkowski \cite{Mat99b,Mat02b,Mat05,Mat09e}, by Matkowski-Pasteczka \cite{MatPas20,MatPas21}, by~Matkowski--P\'ales \cite{MatPal15}, and by Pasteczka \cite{Pas18b,Pas23b,Pas22b}.

There are very close relations between the invariant means and the sequence of iterations of $\MM$. 
For example (cf. \cite{MatPas21}) the $\MM$-invariant mean is uniquely determined if, and only if, for all $x \in I^p$ the sequence of iterates $(\MM^n(x))_{n=1}^\infty$ is convergent to some point of $\Delta_p(I)$. On the other hand, in Borwein-Borwein \cite[Theorem~8.8]{BorBor87}, it was proved that if means are twice continuously differentiable, then the convergence of this iteration is quadratic. In this spirit the author proved in \cite{Pas18b} that in the case when the mean-type mapping consists of quasiarithmetic means satisfying some additional smoothness assumptions, we can efficiently calculate the limit of the sequence of ratios $\Big(\frac{\Var \MM^{n+1}(x)}{\Var \MM^{n}(x)^2}\Big)_{n=1}^\infty$ in the nondegenerated case; see Theorem~\ref{thm:A} above. The limit value turned out to be the variance of Arrow-Pratt indexes of the corresponding means calculated at the point that is equal to the invariant mean. It will be mostly generalized in our next theorem. For the sake of completeness, let us mention that there were slightly different smoothness assumptions in the previous result. 

Similarly to Theorem~\ref{thm:A}, we show that if the mean-type mapping $\MM$ contains residual means only, then the sequence $\Var \MM^n(x)$ of iterations of the vector converges to zero quadratically for >>generic<< vectors $x$ and we calculate the precise speed of this convergence.

\begin{thm}\label{thm:IM}
    Let $p \in \N$, and $\MM:=(M_1,\dots,M_p)\colon I^p\to I^p$ be a mean-type mapping consisting of symmetric, residual means. If there exists a unique $\MM$-invariant mean $K \colon I^p \to I$ then
for each $x \in I^p$ either there exists $n_0\in \N$ such that $\MM^{n_0}(x)$ is constant or
        \Eq{359}{
        \lim_{n \to \infty} \frac{\Var \MM^{n+1}(x)}{(\Var \MM^{n}(x))^2}=\tfrac14\Var\big(\xi_{M_1}\big(K(x)\big),\dots,\xi_{M_p}\big(K(x)\big)\big).
        }
\end{thm}
\begin{proof}
Let $x \in I^p$, and take a compact interval $J \subset I$ such that $x \in J^p$. Define the sequence $(y^{(n)})_{n=0}^\infty$ of the elements in $I^p$ by 
\Eq{*}{
y^{(0)}&:=x;\\
y^{(n)}&:=\MM^{n}(x)\text{ for all }n\ge 1.
}
By the mean property, all $y^{(n)}$-s belong to $J^p$, and therefore we can restrict our consideration to this domain. After this preliminary observation, we unify the definition of residuality for different means. Since all $M_i$-s are residual, for all $i \in\{1,\dots,n\}$, there exist  $\lambda_i,\varepsilon_i \in(0,+\infty)$ and $\alpha_i \in (2,+\infty)$ such that
\Eq{D-resS}{
\abs{M_i(x\vone+s)-x-\EE s-\tfrac12 \xi_{M_i}(x) \Var(s)}\le \lambda_i\norm{s}_{\infty}^{\alpha_i}
}
 is valid for all $x \in J$, and $s \in \mathcal{B}^p(\varepsilon_i)$. We aim to make this inequality $i$-independent. First, we set 
 \Eq{*}{
 \varepsilon_0:=\min\{\varepsilon_1,\dots,\varepsilon_p,1\},\quad \lambda_0:=\max\{\lambda_1,\dots,\lambda_p\},\quad\text{ and }\ 
 \alpha_0:=\min\{\alpha_1,\dots,\alpha_p\}.
 }
 Then $\varepsilon_0 \in(0,1]$, $\lambda_0 \in (0,+\infty)$, and $\alpha_0 \in (2,+\infty)$. Furthermore \eq{D-resS} yields
 \Eq{*}{
\abs{M_i(x\vone+s)-x-\EE s-\tfrac12 \xi_{M_i}(x) \Var(s)}\le \lambda_0\norm{s}_{\infty}^{\alpha_i}
}
to be valid for all $i \in \{1,\dots,p\}$, $x \in J$, and $s \in \mathcal{B}^p(\varepsilon_0)$. Now, since $\norm{s}_{\infty}\le \varepsilon_0\le1$, we can increase the right-hand-side of the latter inequality by replacing $\alpha_i$ by $\alpha_0$. Thus, we have proved the unified version of residuality for $M_i$-s, that is
\Eq{377}{
\abs{M_i(x\vone+s)-x-\EE s-\tfrac12 \xi_{M_i}(x) \Var(s)}\le \lambda_0\norm{s}_{\infty}^{\alpha_0}
}
holds for all $i \in \{1,\dots,p\}$, $x \in J$, and $s \in \mathcal{B}^p(\varepsilon_0)$.

Next, let us set $\mu:=K(x)$, and let $m_n$ be the arithmetic mean of elements in $y^{(n)}$ ($n\in\N$). Then, since $K$ is $\MM$-invariant, we get $\mu=K(y^{(n)})$ for all $n \in\N$.
It is known, see for example \cite{MatPas21}, that the sequence of iterates $(\MM^n)_{n=1}^\infty$ is pointwise convergent to $\KK:=(K,\dots,K)$. Therefore there exists $n_0  \in \N$ such that 
\Eq{*}{
\max y^{(n)}-\min y^{(n)}<\varepsilon_0\text{ for all }n \ge n_0. 
}
Then by the mean property we have $\norm{y^{(n)}-\mu\vone}_\infty < \varepsilon_0$ for all $n \ge n_0$. Thus, we can substitute $x:=\mu$, $s:=s^{(n)}:=y^{(n)}-\mu\vone$ into \eq{377} and obtain
\Eq{*}{
\abs{M_i(y^{(n)})-\EE y^{(n)}-\tfrac12 \xi_{M_i}(\mu) \Var(s^{(n)})}\le\lambda_0\norm{s^{(n)}}_\infty^{\alpha_0},\quad i \in\{1,\dots,p\},\,n\ge n_0.
}
Equivalently
\Eq{*}{
\abs{\EE y^{(n)}-y^{(n+1)}_i+\tfrac12 \xi_{M_i}(\mu) \Var(s^{(n)})}\le\lambda_0\norm{s^{(n)}}_\infty^{\alpha_0},\quad i \in\{1,\dots,p\},\,n\ge n_0.
}
Taking the arithmetic mean with respect to $i$ on the left-hand side and using the triangle inequality, we get
\Eq{*}{
\abs{\EE y^{(n+1)}-\EE y^{(n)}-\tfrac12 \EE \xi_{M_i}(\mu) \Var(s^{(n)})}\le\lambda_0\norm{s^{(n)}}_\infty^{\alpha_0},\quad n\ge n_0.
}
Therefore, binding the above inequatilities, for all $i \in\{1,\dots,p\}$ and $n\ge n_0$ we get
\Eq{*}{
\big|\EE y^{(n+1)}-y_i^{(n+1)} &-\tfrac12 (\EE \xi_{M_j}(\mu) -\xi_{M_i}(\mu)) \Var(s^{(n)})\big|\\
&\le \abs{\EE y^{(n+1)}-\EE y^{(n)}-\tfrac12 \EE \xi_{M_j}(\mu) \Var(s^{(n)})}\\
&\qquad+\abs{\EE y^{(n)}-y^{(n+1)}_i+\tfrac12 \xi_{M_i}(\mu) \Var(s^{(n)})}\le 2\lambda_0\norm{s^{(n)}}_\infty^{\alpha_0}.
}
Since $s^{(n)}$ is a nonconstant vector for all $n\in\N$, we have $\Var(s^{(n)})>0$ and this inequality can rewritten in the following equivalent form
\Eq{434}{
\bigg|\frac{\EE y^{(n+1)}-y_i^{(n+1)}}{\Var(s^{(n)})}-\frac{\EE \xi_{M_j}(\mu) -\xi_{M_i}(\mu)}2 \bigg|\le \sigma_n&\text{ for }i\in\{1,\dots p\}\text{ and }n \ge n_0,
}
 where $\sigma_n:=\frac{2\lambda_0\norm{s^{(n)}}_\infty^{\alpha_0}}{\Var(s^{(n)})}$.
In the next step, we establish the upper bound of $\sigma_n$. First, we majorize the value of the sum of nonnegative elements by the maximal one. Then we
use the easy-to-see property $\max((a-b)^2,(b-c)^2)\ge \frac14 (a-c)^2$ which is valid for all $a,b,c\in \R$. Therefore we have
\Eq{*}{
\Var(s^{(n)})&=\frac{1}{n} \sum_{i=1}^p (s^{(n)}_i-\EE s^{(n)})^2
\ge \frac 1n \max \{ (\max s^{(n)}-\EE s^{(n)})^2, (\min s^{(n)}-\EE s^{(n)})^2\}\\
&\ge \frac{1}{4n}(\max s^{(n)}-\min s^{(n)})^2\qquad (n\in\N).
}
Finally, by the mean value property, all vectors $s^{(n)}$ possess both nonpositive and nonnegative elements, which implies 
$\norm{s^{(n)}}_\infty \le \max s^{(n)}-\min s^{(n)}$ for all $n \in \N$. 
This inequality implies that 
\Eq{*}{
\sigma_n=\frac{2\lambda_0}{\Var(s^{(n)})}\norm{s^{(n)}}_\infty^{\alpha_0} \le 8n\lambda_0\big(\max s^{(n)}-\min s^{(n)}\big)^{\alpha_0-2} \qquad \text{ for all }n \in\N.
}
However, by the invariance principle, we have $\lim\limits_{n \to \infty} \max y^{(n)}=\lim_{n \to \infty} \min y^{(n)}=\mu$, which yields
$\lim\limits_{n \to \infty}  \max s^{(n)}=\lim\limits_{n \to \infty} \min s^{(n)}=0$.
Therefore, since $\alpha_0>2$ and $\sigma_n\ge 0$ for all $n \in \N$, we get
$\lim\limits_{n \to \infty} \sigma_n=0$. Then \eq{434} implies that
\Eq{*}{
\lim_{n \to \infty} \frac{\EE y^{(n+1)}-y_i^{(n+1)}}{\Var(s^{(n)})} &= \frac{\EE \xi_{M_j}(\mu) -\xi_{M_i}(\mu)}2 \qquad \text{for all }i \in\{1,\dots,p\}.
}
Finally we obtain
\Eq{*}{
\lim_{n \to \infty} &\frac{\Var \MM^{n+1}(x)}{(\Var \MM^{n}(x))^2}
=\lim_{n \to \infty}\frac{\Var(y^{(n+1)})}{\Var(y^{(n)})^2}
=\lim_{n \to \infty}\frac{\Var(y^{(n+1)})}{\Var(s^{(n)})^2}\\
&\qquad\qquad=\lim_{n \to \infty}\frac{ \sum_{i=1}^p\big(\EE y^{(n+1)}-y_i^{(n+1)}\big)^2 }{p\Var(s^{(n)})^2}
= \frac1p \sum_{i=1}^p\lim_{n \to \infty}\bigg(\frac{\EE y^{(n+1)}-y_i^{(n+1)}}{\Var(s^{(n)})}\bigg)^2\\
&\qquad\qquad= \frac1p \sum_{i=1}^p \bigg(\frac{\EE \xi_{M_j}(\mu) -\xi_{M_i}(\mu)}{2}\bigg)^2 =\frac{\Var\big(\xi_{M_1}(\mu),\dots,\xi_{M_p}(\mu)\big)}4,
}
which concludes the proof.
\end{proof}

This statement has an immediate corollary.
\begin{cor}\label{cor:IM}
    Let $p \in \N$, and $\MM:=(M_1,\dots,M_p)\colon I^p\to I^p$ be a mean-type mapping consisting of residual, continuous, and strict means.
For each $x \in I^p$ either there exists $n_0\in \N$ such that $\MM^{n_0}(x)$ is constant or \eq{359} holds,         where $K \colon I^p \to I$ is the (unique) $\MM$-invariant mean.
\end{cor}

\section{Examples}
This section aims to establish residua for possibly broad classes of means. This allows one to apply Theorem~\ref{thm:IM} for all mean-type mappings built from means in these classes. We intentionally avoid direct applications of Theorem~\ref{thm:IM}, since it is a purely technical operation.

\subsection{\label{sec:QD}Quasideviation means}
The notions of deviations and quasideviations were introduced by Daróczy \cite{Dar71b} and by Páles \cite{Pal82a}, respectively. In what follows, we recall Definition~2.1 and Theorem~2.1 from \cite{Pal82a}. A bivariate function $E:I^2\to\R$ is called a \emph{quasideviation} if $E$ has the following three properties:
\begin{enumerate}[(D1)]
 \item For all $x,u\in I$, the equality $\sign E(x,u)=\sign(x-u)$ holds;
 \item For all $x\in I$, the mapping $I\ni u\mapsto E(x,u)$ is continuous;
 \item For all $x,y\in I$ with $x<y$, the mapping
 \Eq{*}{
   (x,y)\ni u\mapsto\frac{E(x,u)}{E(y,u)}
 }
 is strictly decreasing.
\end{enumerate}
In order to introduce quasideviation means, the following statement is instrumental (cf.\ \cite[Theorem 2.1]{Pal82a}). 

\begin{prop} Given a quasideviation $E:I^2\to\R$, a number $n\in \N$, and points $x_1,\dots,x_n\in I$. There exists a unique element $u\in I$ such that
\Eq{E(u)}{
  E(x_1,u)+\dots+E(x_n,u)=0.
}
Furthermore, $\min(x_1,\dots,x_n)<u<\max(x_1,\dots,x_n)$ unless $x_1=\dots=x_n$.
\end{prop}

For $n\in\N$ and $x_1,\dots,x_n\in I$, the solution $u$ of equation \eq{E(u)} is called the \emph{\mbox{$E$-quasideviation} mean of $x_1,\dots,x_n$} and will be denoted by $\D_E(x_1,\dots,x_n)$. In view of \cite[Characterization theorem 2]{Pal82a}, we know that quasideviation means are symmetric and repetition invariant.

We say that a quasideviation $E\colon I^2\to\R$ is \emph{normalizable} (cf.\ \cite{Dar71b}) if, for all $x\in I$, the function $u\mapsto E(x,u)$ is differentiable at $x$ and the mapping $x\mapsto\partial_2E(x,x)$ is negative and continuous on $I$. 


\begin{thm}\label{thm:QD}
Let $E \in \mathcal{C}^2(I^2)$ be a normalizable quasideviation. Then $\D_E$ is twice continuously differentiable. Moreover
\Eq{*}{
    \xi_{\D_E}(x)= \frac{\partial_1^2E}{\partial_1E}(x,x),\qquad x \in I.
}
\end{thm}
\begin{proof}
By the definition of quasideviation mean, in view of the implicit function theorem, we know that $\D_E \in \mathcal{C}^2$.  

Fix $p \in \N$, $x \in I$ and take any $\varepsilon_0>0$ such that $(x-\varepsilon_0,x+\varepsilon_0)\subset I$. For all $\varepsilon \in (-\varepsilon_0,\varepsilon_0)\setminus\{0\}$ we define
\Eq{*}{
\delta=\delta(\varepsilon):=\D_E(x+\varepsilon,\underbrace{x,\dots,x}_{(p-1) \text{entries}})-x.
}
Then
\Eq{*}{
x+\delta=\D_E(x+\varepsilon,\underbrace{x,\dots,x}_{(p-1) \text{entries}}).
}
Therefore, by the definition of quasideviation mean, we get that $\delta$ is the unique solution of the equation
\Eq{500}{
0&=E(x+\varepsilon,x+\delta)+(p-1)E(x,x+\delta).
}
Now we use Taylor expansion of $E$ twice. Indeed, by property (D1), we have
\Eq{*}{
0=E(x+\varepsilon,x+\varepsilon)&=E(x,x)+(\partial_1E +\partial_2E)(x,x) \varepsilon+(\partial_1^2E+2\partial_1\partial_2E+\partial_2^2E)(x,x)\varepsilon^2.
}
Thus we get 
\Eq{509}{
(\partial_1E +\partial_2E)(x,x)=0\quad\text{and}\quad (\partial_1^2E+2\partial_1\partial_2E+\partial_2^2E)(x,x)=0 \qquad (x \in I).
}

Now we apply Taylor's theorem for the second time. The expansion of \eq{500} at $(\varepsilon,\delta)=(0,0)$ gives us the following equality
\Eq{*}{
0&=\partial_1E (x,x) \varepsilon + \partial_2E(x,x)\delta+\tfrac12  \partial_1^2E(x,x) \varepsilon^2+\tfrac12  \partial_2^2E(x,x) \delta^2+\partial_1\partial_2E(x,x) \varepsilon\delta\\
&\quad+(p-1)\big(\partial_2E(x,x) \delta +\tfrac12 \partial_2^2E(x,x) \delta^2\big)+o((\varepsilon+\delta)^2).
\\
}

Taking the first equality of \eq{509} into account, we have
\Eq{E334}{
\big( p\delta -\varepsilon \big)\partial_1E(x,x)&=\tfrac12\big(\partial_1^2E(x,x)\varepsilon^2+2\partial_1\partial_2E(x,x)\varepsilon\delta+p \partial_2^2E(x,x)\delta^2\big)\\
&\qquad+o((\delta+\varepsilon)^2).
}
Now define $c\colon (-\varepsilon_0,\varepsilon_0) \setminus \{0\} \to \R$ by $c(\varepsilon):=\frac{1}{p\varepsilon^2}(p\delta(\varepsilon)-\varepsilon)$.
Then we have  
\Eq{339}{
\delta(\varepsilon)=\tfrac\varepsilon p+\varepsilon^2c(\varepsilon).
}
Hence, by Lemma~\ref{lem:Formxi_M}, we have
\Eq{*}{
\xi_{\D_E}(x)
&=\lim_{\varepsilon \to 0} \frac{2p^2}{(p-1)\varepsilon^2}(\D_E(x+\varepsilon,\underbrace{x,\dots,x}_{(p-1) \text{entries}})-x -\tfrac \varepsilon p)\\
&= \frac{2p^2}{p-1}\lim_{\varepsilon \to 0} \dfrac{1}{\varepsilon^2}(\delta(\varepsilon) -\tfrac \varepsilon p)
=\frac{2p^2}{p-1}\lim_{\varepsilon \to 0} c(\varepsilon).
}
Thus $c$ has a finite limit at $0$, say $c_0 \in \R$, and 
\Eq{544}{
\xi_{\D_E}(x)=\frac{2p^2}{p-1}c_0.
}
Now we calculate $c_0$ differently. 
Pluging \eq{339} to \eq{E334} gives us
\Eq{*}{
p\varepsilon^2c(\varepsilon)\partial_1E(x,x)&=\frac{\varepsilon^2}2\Big(\partial_1^2E(x,x)+2\partial_1\partial_2E(x,x)\big(\tfrac1p+\varepsilon c(\varepsilon)\big)+p \partial_2^2E(x,x)\big(\tfrac1p+\varepsilon c(\varepsilon)\big)^2\Big)\\
&\quad +o(\varepsilon^2).
}

Thus
\Eq{*}{
c(\varepsilon)&=
\frac1{2p}\frac{\partial_1^2E(x,x)}{\partial_1E(x,x)}+\frac{\partial_1\partial_2E(x,x)}{{\partial_1E(x,x)}}\frac{\tfrac1p+\varepsilon c(\varepsilon)}{p}+ \frac{\partial_2^2E(x,x)}{\partial_1E(x,x)}\frac{\big(\tfrac1p+\varepsilon c(\varepsilon)\big)^2}{2}+o(1).
}
Upon passing the limit $\varepsilon \to 0$, since $\lim_{\varepsilon \to 0}c(\varepsilon)=c_0$ is a (finite) real number, we have 
\Eq{559}{
c_0&=\frac1{2p}\frac{\partial_1^2E(x,x)}{\partial_1E(x,x)}+\frac{\partial_1\partial_2E(x,x)}{{\partial_1E(x,x)}}\frac{1}{p^2}+ \frac{\partial_2^2E(x,x)}{\partial_1E(x,x)}\frac{1}{2p^2}\\
&=\frac{\partial_2^2E(x,x)+2\partial_1\partial_2E(x,x)+p\partial_1^2E(x,x)}{2p^2\partial_1E(x,x)}.\\
}
Now we use the second equality of \eq{509} which, after easy transformation, is equivalent to  $\partial_2^2E(x,x)+2\partial_1\partial_2E(x,x)=-\partial_1^2E(x,x)$. Finally, binding \eq{544}, \eq{559}, and this property, we obtain
\Eq{*}{
\xi_{\D_E}(x)=\frac{2p^2}{p-1}c_0&=\frac{2p^2}{p-1}\cdot \frac{\partial_2^2E(x,x)+2\partial_1\partial_2E(x,x)+p\partial_1^2E(x,x)}{2p^2\partial_1E(x,x)}\\
&=\frac{2p^2}{p-1}\cdot \frac{-\partial_1^2E(x,x)+p\partial_1^2E(x,x)}{2p^2\partial_1E(x,x)}=\frac{\partial_1^2E(x,x)}{\partial_1E(x,x)},
}
which completes the proof.
\end{proof}

\subsection{Bajraktarevi\'c means} Now we study an important subclass of quasideviation means defined by Bajraktarević in the papers \cite{Baj58,Baj63}. 
Let $f,g\colon I\to\R$ be two functions such that $g>0$ and the ratio $x \mapsto \frac{f(x)}{g(x)}$ is continuous and strictly increasing. Then \Eq{QDforB}{
E\colon I \times I \ni (x,u)\mapsto g(u)f(x)-f(u)g(x)
}
is a quasideviation (cf. \cite{Pal87d}), and a simple computation yields that 
\Eq{*}{
\B_{f,g}(x):=\D_E(x)=\Big(\frac fg\Big)^{-1}\left(\frac{f(x_1)+\cdots+f(x_p)}{g(x_1)+\cdots+g(x_p)}\right) \qquad (x\in I^p).
}

Before we proceed to calculate the value of residuum let us recall the theorem which characterizes the equality of Bajraktarevi\'c means. Namely, for the functions $f,g \in \mathcal{C}^2(I)$, P\'ales-Zakaria \cite{PalZak20b} defined two auxiliary operators $\Phi_{f,g},\Psi_{f,g}\colon I \to \R$ by
\Eq{*}{
\Phi_{f,g}(x):=\bigg(\frac{gf''-fg''}{gf'-fg'}\bigg)(x),\qquad 
\Psi_{f,g}(x):=\bigg(\frac{g'f''-f'g''}{gf'-fg'}\bigg)(x);
}
and proved that for $f,g,h,k \colon I \to \R$ (satisfying certain smoothness assumptions) the equality $\B_{f,g}=\B_{h,k}$ holds if and only if $\Phi_{f,g}=\Phi_{h,k}$ and $\Psi_{f,g}=\Psi_{h,k}$ (cf. \cite[Theorem~2.1]{PalZak20b}). 

We show that the function $\Phi_{f,g}$ is the residuum of $\B_{f,g}$. We can combine this fact with the above equality conditions to construct Bajraktarevi\'c means which have the same residuum, although they are not equal to each other. Remarkably, if the mean-type mapping consists of means with the same residuum then, in view of Theorem~\ref{thm:main}, the convergence of (variances of) its iterates is superqudratic.

\begin{prop}\label{prop:Baja}
Let $f,g \in \mathcal{C}^2(I)$ such that $g>0$, the ratio $x \mapsto \frac{f(x)}{g(x)}$ is strictly increasing, and the mapping $gf'-fg'$ is nowhere vanishing.
Then $\xi_{\B_{f,g}}=\Phi_{f,g}$.
\end{prop}
\begin{proof}
We know that $\B_{f,g}=\D_E$, where $E \colon I \times I \to \R$ is of the form \eq{QDforB}. Moreover
\Eq{*}{
\partial_1E(x,u)=g(u)f'(x)-f(u)g'(x)\text{ and }
\partial_1^2E(x,u)=g(u)f''(x)-f(u)g''(x).
}
Therefore, by Theorem~\ref{thm:QD}, for all $x \in I$, we have
\Eq{*}{
\xi_{\B_{f,g}}(x)=\xi_{\D_E}(x)= \frac{\partial_1^2E}{\partial_1E}(x,x)= \frac{g(x)f''(x)-f(x)g''(x)}{g(x)f'(x)-f(x)g'(x)}=\Phi_{f,g}(x),
}
which completes the proof.
\end{proof}
\subsection{Gini means}
Now we study a generalization of H\"older means introduced by Gini \cite{Gin38}. This is simultaneously a subclass of Bajraktarevi\'c means.

More precisely, for $\alpha,\beta\in\R$, define the \emph{Gini mean} $\G_{\alpha,\beta}\colon \bigcup_{n=1}^{\infty} \R_+^n \to \R_+$ by
\Eq{*}{
\G_{\alpha,\beta} (x):= 
\begin{cases} 
\left(\dfrac{x_1^\alpha+\cdots+x_n^\alpha}{x_1^\beta+\cdots+x_n^\beta} \right)^{\frac{1}{\alpha-\beta}} &\quad \text{ if } \alpha \ne \beta, \\[4mm]        \exp\left(\dfrac{x_1^\alpha\log x_1+\cdots+x_n^\alpha\log x_n}{x_1^\alpha+\cdots+x_n^\alpha} \right) &\quad \text{ if } \alpha = \beta.
\end{cases}
}
It is easy to check that, in the particular case $\beta=0$, we get the $\alpha$-th power (or H\"older) mean $\H_\alpha=\G_{\alpha,0}$. Furthermore $\G_{\alpha,\beta}=\G_{\beta,\alpha}$ for all $\alpha,\beta \in \R$. In addition, for a given $\alpha,\beta\in\R$ we set
\Eq{*}{
f(x):=\begin{cases}
    x^\alpha &\text{ if }\alpha \ne \beta,\\
    x^\alpha\log(x)&\text{ if }\alpha = \beta,
\end{cases}
\qquad\text{ and }\qquad g(x):=x^\beta,
}
and we get $\B_{f,g}=\G_{\alpha,\beta}$. We can utilize this fact to establish the residuum of all means belonging to this family.
\begin{prop}
For all $\alpha,\beta \in \R$ we have
$\xi_{\G_{\alpha,\beta}}(x)= \frac{\alpha+\beta-1}{x}$.
\end{prop}
\begin{proof}
In view of Proposition~\ref{prop:Baja}, for $\alpha \ne \beta$ and $x \in \R_+$ we have 
    \Eq{*}{
  \xi_{\G_{\alpha,\beta}}(x)&= \frac{x^\beta(x^\alpha)''-x^\alpha(x^\beta)''}{(x^\beta(x^\alpha)'-x^\alpha(x^\beta)'}=\frac{(\alpha(\alpha-1)-\beta(\beta-1))x^{\alpha+\beta-2}}{(\alpha-\beta)x^{\alpha+\beta-1}}\\
&=\frac{(\alpha-\beta)(\alpha+\beta-1)}{(\alpha-\beta)x}=\frac{\alpha+\beta-1}{x}.
    }
    Similarly, in the case $\alpha=\beta$, we obtain
    $\xi_{\G_{\alpha,\alpha}}(x)=\frac{2\alpha-1}x=\frac{\alpha+\beta-1}x$.
\end{proof}

\begin{xrem}
Observe that, in view of the above proposition, for every $\alpha \in \R$, all Gini means belonging to the family $\{\G_{c,\alpha-c}\colon c \in \R\}$ have the same residuum. 
\end{xrem}
\subsection{\label{sec:QA}Quasiarithmetic means}
Another important generalization of H\"older's means is the notion of quasiarithmetic means. This family was introduced in the series of several simultaneous papers \cite{Kno28,Def31,Kol30,Nag30} in 1920-s/30-s as a generalization of power means. Given a continuous strictly monotone function $f \colon I \to \R$, the \emph{quasiarithmetic mean} 
$\QA{f}:\bigcup_{n=1}^{\infty} I^n \to I$ is defined by
\Eq{QA}{
\QA{f}(x):= f^{-1} \left( \frac{f(x_1)+\cdots+f(x_n)}{n} \right).
}
If $\alpha\in\R\setminus\{0\}$ and $f(x):=x^\alpha$ for $x\in\R_+$, then $\QA{f}=\H_\alpha$. If $f(x):=\log x$ for $x\in\R_+$, then $\QA{f}=\H_0$, therefore, H\"older means are quasiarithmetic means.

If $g$ is a constant function, then one can see that $\B_{f,g}=\QA{f}$ and hence quasiarithmetic means are also included in the class of Bajraktarević means. Then, as an immediate result of Proposition~\ref{prop:Baja} we can establish the following statement.
\begin{prop}
Let $f\in \mathcal{C}^2(I)$ be a continuous and strictly increasing function with nowhere vanishing first derivative.
Then $\QA{f}$ is twice continuously differentiable, and
\Eq{*}{
\xi_{\QA{f}}(x)= \frac{f''(x)}{f'(x)},\qquad x \in I.
}
\end{prop}

In view of the latter proposition we see that, in the case when all means $M_i$ are quasiarithmetic means, Theorem~\ref{thm:IM} simplifies to Theorem~\ref{thm:A} (up to certain issues related to the smoothness assumptions).

\subsection{An example} In the last section we show a classical application of our statement in the easy example.
    Let $\MM \colon \R_+^3 \to \R_+^3$ be given by
    \Eq{*}{
    \MM(x_1,x_2,x_3):=\bigg(\frac{x_1^2+x_2^2+x_3^2}{x_1+x_2+x_3},\sqrt[3]{x_1x_2x_3},\sqrt{\frac{x_1+x_2+x_3}{x_1^{-1}+x_2^{-1}+x_3^{-1}}}\bigg).
    }
    Then $\MM=(\G_{2,1},\G_{0,0},\G_{1,-1})$. Since Gini means are strict, by \cite[Theorem~8.8]{BorBor87}, there exists a uniquely determined $\MM$-invariant mean $K \colon \R_+^3\to \R_+$. Using the classical result concerning the comparability of Gini means (cf. \cite{Los71a,Los71c}), one can show that $\MM^{n}(x)$ is nonconstant for every vector $x\in \R_+^3$ and $n \in \N$.
    Moreover
    \Eq{*}{
    \xi_{\G_{2,1}}(t)=\tfrac2t,\quad\xi_{\G_{0,0}}(t)=-\tfrac{1}t,\quad\xi_{\G_{1,-1}}(t)= -\tfrac1t \qquad (t \in \R_+).
    }
    Hence 
    $    \Var\big(\xi_{\G_{2,1}}(t),\xi_{\G_{0,0}}(t),\xi_{\G_{1,-1}}(t)\big)=\tfrac{\Var(2,-1,-1)}{t^2}=\tfrac{2}{t^2}.
    $
    Thus, by Theorem~\ref{thm:IM}, we have 
    \Eq{*}{
            \lim_{n \to \infty} \frac{\Var \MM^{n+1}(x)}{(\Var \MM^{n}(x))^2}=\tfrac14\Var\big(\xi_{M_1}\big(K(x)\big),\dots,\xi_{M_p}\big(K(x)\big)\big)=\frac{1}{2(K(x))^2}.
    }

\section*{Concluding remarks} In this section, we would like to emphasize few final (summarizing) remarks:

\begin{enumerate}[1.]
\item Residuality is an intermediate property of symmetric means between being $\mathcal{C}^2$ and $\mathcal{C}^3$.
\item Binding two results \cite[Lemma~4.1]{Pas16b} and \cite[Lemma~2.2]{Pas18b} we can show that quasiarithmetic means generated by $\mathcal{C}^2$ functions with nowhere vanishing first derivative and such that $f''$ is of almost bounded variation (finite variation restricted to every compact interval; cf. \cite[p. 135]{KucChoGer90}) is a residual mean.
\item By \eq{D-res}, the value of $\xi_M$ measures the distance between $M$ and the arithmetic mean in the vicinity of the diagonal.
\item We proved three ways of calculating the residuum of a mean. First two are given in \eq{xiM} while the third one is presented in Lemma~\ref{lem:Formxi_M}.
\item Every residual means has a well-defined residuum. The converse is not known, that is, it could be that a mean of the form \eq{E1} is not residual. 
\item This study is closely related to the definition of local comparison of means \cite{PalZak17}. In fact, it is easy to observe that the strict inequality between residua implies the comparability of means in the vicinity of the diagonal.
\item We abandoned the study of the nature of the invariant mean. However, a few natural questions appear:
\begin{enumerate}
\item Is it true that if a mean-type mapping consists of residual means, then the invariant mean is also residual?
\item If so, can we calculate the residuum of the invariant mean based solely on the input mean residua?
\end{enumerate}
\item We can generalize the first question to the following one. Assume that all means in the mean-type mapping satisfy a certain smoothness assumption (say they are all $\mathcal{C}^k$). Does it imply that the same smoothness assumption is valid for the invariant mean?
\end{enumerate}


\begin{thebibliography}{10}

\bibitem{Baj58}
M.~Bajraktarević.
\newblock {Sur une équation fonctionnelle aux valeurs moyennes}.
\newblock {\em Glasnik Mat.-Fiz. Astronom. Društvo Mat. Fiz. Hrvatske Ser.
  II}, 13:243–248, 1958.

\bibitem{Baj63}
M.~Bajraktarević.
\newblock {Sur une généralisation des moyennes quasilinéaires}.
\newblock {\em Publ. Inst. Math. (Beograd) (N.S.)}, 3 (17):69–76, 1963.

\bibitem{BajPal09b}
Sz. Baják and Zs. Páles.
\newblock {Computer aided solution of the invariance equation for two-variable
  {G}ini means}.
\newblock {\em Comput. Math. Appl.}, 58:334–340, 2009.

\bibitem{BajPal09a}
Sz. Baják and Zs. Páles.
\newblock {Invariance equation for generalized quasi-arithmetic means}.
\newblock {\em Aequationes Math.}, 77:133–145, 2009.

\bibitem{BajPal10}
Sz. Baják and Zs. Páles.
\newblock {Computer aided solution of the invariance equation for two-variable
  {S}tolarsky means}.
\newblock {\em Appl. Math. Comput.}, 216(11):3219–3227, 2010.

\bibitem{BajPal13}
Sz. Baják and Zs. Páles.
\newblock {Solving invariance equations involving homogeneous means with the
  help of computer}.
\newblock {\em Appl. Math. Comput.}, 219(11):6297–6315, 2013.

\bibitem{BorBor87}
J.~M. Borwein and P.~B. Borwein.
\newblock {\em {Pi and the {AGM}}}.
\newblock {Canadian Mathematical Society Series of Monographs and Advanced
  Texts}. John Wiley \& Sons, Inc., New York, 1987.
\newblock A study in analytic number theory and computational complexity, A
  Wiley-Interscience Publication.

\bibitem{Dar71b}
Z.~Daróczy.
\newblock {A general inequality for means}.
\newblock {\em Aequationes Math.}, 7(1):16–21, 1971.

\bibitem{Dar05a}
Z.~Daróczy.
\newblock {Functional equations involving means and {G}auss compositions of
  means}.
\newblock {\em Nonlinear Anal.}, 63(5-7):e417–e425, 2005.

\bibitem{DarPal02c}
Z.~Daróczy and Zs. Páles.
\newblock {Gauss-composition of means and the solution of the
  {M}atkowski–{S}utô problem}.
\newblock {\em Publ. Math. Debrecen}, 61(1-2):157–218, 2002.

\bibitem{DarPal03a}
Z.~Daróczy and Zs. Páles.
\newblock {The {M}atkowski–{S}utô problem for weighted quasi-arithmetic
  means}.
\newblock {\em Acta Math. Hungar.}, 100(3):237–243, 2003.

\bibitem{Def31}
B.~de~Finetti.
\newblock {{S}ul concetto di media}.
\newblock {\em Giornale dell' Instituto, Italiano degli Attuarii}, 2:369–396,
  1931.

\bibitem{FosPhi84a}
D.~M.~E. Foster and G.~M. Phillips.
\newblock {The arithmetic-harmonic mean}.
\newblock {\em Math. Comp.}, 42(165):183–191, 1984.

\bibitem{Gau18}
C.~F. Gauss.
\newblock {Nachlass: Aritmetisch-geometrisches Mittel}.
\newblock In {\em {Werke 3 (Göttingem 1876)}}, page 357–402. Königliche
  Gesellschaft der Wissenschaften, 1818.

\bibitem{Gin38}
C.~Gini.
\newblock {{D}i una formula compressiva delle medie}.
\newblock {\em Metron}, 13:3–22, 1938.

\bibitem{Gla11b}
D.~Głazowska.
\newblock {A solution of an open problem concerning {L}agrangian mean-type
  mappings}.
\newblock {\em Cent. Eur. J. Math.}, 9(5):1067–1073, 2011.

\bibitem{Gla11a}
D.~Głazowska.
\newblock {Some {C}auchy mean-type mappings for which the geometric mean is
  invariant}.
\newblock {\em J. Math. Anal. Appl.}, 375(2):418–430, 2011.

\bibitem{HarLitPol34}
G.~H. Hardy, J.~E. Littlewood, and G.~Pólya.
\newblock {\em {Inequalities}}.
\newblock Cambridge University Press, Cambridge, 1934.
\newblock (first edition), 1952 (second edition).

\bibitem{JarJar18}
J.~Jarczyk and W.~Jarczyk.
\newblock {Invariance of means}.
\newblock {\em Aequationes Math.}, 92(5):801–872, 2018.

\bibitem{Kno28}
K.~Knopp.
\newblock {Über {R}eihen mit positiven {G}liedern}.
\newblock {\em J. London Math. Soc.}, 3:205–211, 1928.

\bibitem{Kol30}
A.~N. Kolmogorov.
\newblock {{S}ur la notion de la moyenne}.
\newblock {\em Rend. Accad. dei Lincei (6)}, 12:388–391, 1930.

\bibitem{KucChoGer90}
M.~Kuczma, B.~Choczewski, and R.~Ger.
\newblock {\em {Iterative {F}unctional {E}quations}}, volume~32 of {\em
  {Encyclopedia of Mathematics and its Applications}}.
\newblock Cambridge University Press, Cambridge, 1990.

\bibitem{Lag84}
J.L. Lagrange.
\newblock Sur une nouvelle m\'ethode de calcul int\'egrale pour
  diff\'erentielles affect\'ees d’un radical carre.
\newblock {\em Mem. Acad. R. Sci. Turin II}, 2:252–312, 1784-1785.

\bibitem{Leh71}
D.~H. Lehmer.
\newblock {{O}n the compounding of certain means}.
\newblock {\em J. Math. Anal. Appl.}, 36:183–200, 1971.

\bibitem{Los71a}
L.~Losonczi.
\newblock {Subadditive {M}ittelwerte}.
\newblock {\em Arch. Math. (Basel)}, 22:168–174, 1971.

\bibitem{Los71c}
L.~Losonczi.
\newblock {Subhomogene {M}ittelwerte}.
\newblock {\em Acta Math. Acad. Sci. Hungar.}, 22:187–195, 1971.

\bibitem{Mat99b}
J.~Matkowski.
\newblock {Iterations of mean-type mappings and invariant means}.
\newblock {\em Ann. Math. Sil.}, (13):211–226, 1999.
\newblock European Conference on Iteration Theory (Muszyna-Złockie, 1998).

\bibitem{Mat02b}
J.~Matkowski.
\newblock {On iteration semigroups of mean-type mappings and invariant means}.
\newblock {\em Aequationes Math.}, 64(3):297–303, 2002.

\bibitem{Mat05}
J.~Matkowski.
\newblock {Lagrangian mean-type mappings for which the arithmetic mean is
  invariant}.
\newblock {\em J. Math. Anal. Appl.}, 309(1):15–24, 2005.

\bibitem{Mat09e}
J.~Matkowski.
\newblock Iterations of the mean-type mappings.
\newblock In {\em Iteration theory ({ECIT} '08)}, volume 354 of {\em Grazer
  Math. Ber.}, pages 158--179. Institut f\"{u}r Mathematik,
  Karl-Franzens-Universit\"{a}t Graz, Graz, 2009.

\bibitem{MatPas20}
J.~Matkowski and P.~Pasteczka.
\newblock {Invariant means and iterates of mean-type mappings}.
\newblock {\em Aequationes Math.}, 94(3):405–414, 2020.

\bibitem{MatPas21}
J.~Matkowski and P.~Pasteczka.
\newblock {Mean-type mappings and invariance principle}.
\newblock {\em Math. Inequal. Appl.}, 24(1):209–217, 2021.

\bibitem{MatPal15}
J.~Matkowski and Zs. Páles.
\newblock {Characterization of generalized quasi-arithmetic means}.
\newblock {\em Acta Sci. Math. (Szeged)}, 81(3-4):447–456, 2015.

\bibitem{Mik48}
Jan~G. Mikusiński.
\newblock {Sur les moyennes de la forme {$\psi{}^{-1}[\sum q\psi{}(x)]$}}.
\newblock {\em Studia Math.}, 10:90–96, 1948.

\bibitem{Nag30}
M.~Nagumo.
\newblock {Über eine {K}lasse der {M}ittelwerte}.
\newblock {\em Japanese J. Math.}, 7:71–79, 1930.

\bibitem{Pas16b}
P.~Pasteczka.
\newblock {Iterated quasi-arithmetic mean-type mappings}.
\newblock {\em Colloq. Math.}, 144(2):215–228, 2016.

\bibitem{Pas18b}
P.~Pasteczka.
\newblock {On the quasi-arithmetic {G}auss-type iteration}.
\newblock {\em Aequationes Math.}, 92(6):1119–1128, 2018.

\bibitem{Pas22b}
P.~Pasteczka.
\newblock {There is at most one continuous invariant mean}.
\newblock {\em Aequationes Math.}, 96(4):833–841, 2022.

\bibitem{Pas23b}
P.~Pasteczka.
\newblock Invariance property for extended means.
\newblock {\em Results Math.}, 78(1):Paper No. 146, 2023.

\bibitem{Pal82a}
Zs. Páles.
\newblock {Characterization of quasideviation means}.
\newblock {\em Acta Math. Acad. Sci. Hungar.}, 40(3-4):243–260, 1982.

\bibitem{Pal87d}
Zs. Páles.
\newblock {On the characterization of quasi-arithmetic means with weight
  function}.
\newblock {\em Aequationes Math.}, 32(2-3):171–194, 1987.

\bibitem{PalPas16}
Zs. Páles and P.~Pasteczka.
\newblock {Characterization of the {H}ardy property of means and the best
  {H}ardy constants}.
\newblock {\em Math. Inequal. Appl.}, 19(4):1141–1158, 2016.

\bibitem{PalZak17}
Zs. Páles and A.~Zakaria.
\newblock {On the local and global comparison of generalized {B}ajraktarević
  means}.
\newblock {\em J. Math. Anal. Appl.}, 455(1):792–815, 2017.

\bibitem{PalZak20b}
Zs. Páles and A.~Zakaria.
\newblock {Equality and homogeneity of generalized integral means}.
\newblock {\em Acta Math. Hungar.}, 160(2):412–443, 2020.

\bibitem{Sch82}
I.~J. Schoenberg.
\newblock {\em {Mathematical time exposures}}.
\newblock Mathematical Association of America, Washington, DC, 1982.

\end{thebibliography}

\end{document}